\documentclass[11pt]{article}
\usepackage{srcltx}

\usepackage{amsmath,amsfonts,latexsym}
\usepackage{latexsym,amsfonts,euscript,amssymb}
\numberwithin{equation}{section}

\newtheorem{theorem}{Theorem}[section]

\newtheorem{lemma}[theorem]{Lemma}

\newcommand{\cy}[1]{\mbox{$\langle #1 \rangle $}}

\newenvironment{proof}{\underline{Proof:}}{ \hfill $\Box$ 
                      \vspace{\baselineskip}}

\begin{document}

\pagestyle{myheadings}

\title{Finite groups with small number of cyclic subgroups}

\author{Wei Zhou}

\date{}
\maketitle

\begin{center}
	School of Mathematics and Statistics,\\
	Southwest University, Chongqing 400715, P. R. CHINA\\
	zh\_great@swu.edu.cn	
\end{center}

\footnotetext{Support by
	National Natural Science Foundation of China (Grant No. 11471266). 
	}

\renewcommand{\arraystretch}{0.1}

\begin{abstract}
	In this note, we study the finite groups with the number of cylic subgroups no greater than 5.
\end{abstract}

\baselineskip .65 true cm

\section{Introduction}
For a finite group $G$, let $C(G)$ be the poset of cyclic subgroups of $G$. Sometimes $C(G)$ can decide the structure of $G$. For example, $|C(G)|=|G|$ if and only if $G$ is an elementary abelian $2$-group. 
The groups $G$ such that $|G|-|C(G)|\le 3$ is classified in \cite{Tar1, Tar2} and \cite{Zhou}.

It is well-known that a finite $p$-group $G$ has eactly one cyclic subgroup of order $p$ if and only if $G$ is cyclic or generalized quaternion group. Hence the groups with small number of cyclic subgroups will be interesting. 
In this note, we shall study the finite group $G$ with $|C(G)|\le 5$. 

It is esay to see that $|C(G)|=1$ if and only $G=1$, and $|C(G)|=2$ if and only if $G\cong C_p$ for some prime $p$.
In this note, we will focus on the group $G$ such that $3\le |C(G)|\le 5$.

For a finite group $G$, denote by $\pi_e(G)$ the set of all element orders of $G$, and by $\pi(G)$ the set of all prime divisors of $|G|$. For any $i \in \pi_e(G)$, denote by $C_i(G)$ the set of all cylic subgroups of order $i$ in $G$. Throughout this note, let $c_i=|C_i(G)|$. 

\section{The main result}

For a finite group $G$, we know 
	\begin{equation} \label{eq:1}
		\begin{split}
			&|G|=\sum_{k\in \pi_e(G)} c_k\cdot\phi (k), \\
			&|C(G)|=\sum_{k \in \pi_e(G)} c_k,
		\end{split}
	\end{equation}
	where $\phi$ is the Eucler function.

\begin{lemma} \label{3C}
	If $|C(G)|=3$, then $G\cong C_{p^2}$ for some prime $p$.
\end{lemma}
\begin{proof}
	We claim that $|\pi(G)|=1$. Otherwise, by Cauchy theorem, $|\pi(G)|=2$. Let $|G|=p^aq^b$, where $a, b \ge 1$. By equation \ref{eq:1}, $c_p=c_q=1$. Hence there exist $A \lhd G$ and $B\lhd G$ such that $A \cong C_p$ and $B\cong C_q$. Thus $AB\cong C_{pq}$. Note that $|C(C_{pq})|=4$, a contradiction.
	
	Hence $G$ is a $p$-groups. Let $|G|=p^n$. By equation \ref{eq:1}, $c_p=1$ or $2$. If $c_p=2$, we get that $p^n=2p+1$, a contradiction. Hence $c_p=1$. Then $c_{p^2}=1$. From equation \ref{eq:1}, $n=2$ and $G \cong C_{p^2}$. 
\end{proof}

\begin{lemma} \label{4C}
	If $|C(G)|=4$, then $G \cong C_{pq}, C_{p^3}$.
\end{lemma}
\begin{proof}
	From equation \ref{eq:1}, we see that $|\pi(G)|\le 3$. If  $\pi(G)=\{p, q, r\}$, then $c_p=c_q=c_r=1$ and there exist
	$A, B, C \lhd G$ such that $A\cong C_p$, $B\cong C_q$ and $C\cong C_r$. Thus $ABC\cong C_{pqr}$. Note that $|C(C_{pqr})|>4$, a contradiction. Now we get the following two cases:
	
	Case 1. $\pi(G)=\{p, q\}$. Let $|G|=p^aq^b$. From equation \ref{eq:1}, we have that $c_p=2$ and $ c_q=1$ or $c_p=1=c_q$.
	If $c_p=2$ and $ c_q=1$, then $p^aq^b=2p+q-2$. Since $p^aq^b-2p-q+2\ge pq-2p-q+2=(p-1)(q-2)$, we get $q=2$ and $a=b=1$. Since $c_2=1$, we get $G\cong C_2\times C_q$, contrary to $c_p=2$. Therefore $c_p=c_q=1$, and we can find $A, B\lhd G$ such that $A\cong C_p$ and $B \cong C_q$. Note that $AB\cong C_{pq}$ and $|C(AB)|=4$. It follows that $G=AB\cong C_{pq}$.
	
	Case 2. $\pi(G)=\{p\}$. Let $|G|=p^n$. We know that $c_p\le 3$. If  $c_p=3$, then $p^n=1+3(p-1)$. Hence $p=2$ and $n=2$. This is impossible. If $c_p=2$, then $c_{q^2}=1$, and $p^n=p^2+p-1$. Note that $p^n-p^2-p+1\ge p^3-p^2-p+1=(p-1)(p^2-1)>0$. This is impossible. So we get that $c_p=1$. By \cite[Satz 3.8.2]{Hup}, $G$ is a cyclic or $p=2$ and $G$ is a generalized quaternion group. Clearly,  we get $G \cong C_{p^3}$ in this case.	
\end{proof}

\begin{lemma}\label{5C}
	If $|C(G)|=5$, then $G\cong S_3, C_{p^4}, C_3 \times C_3, Q_8$.
\end{lemma}
\begin{proof}
	Let $\pi(G)=\{p_1, \cdots, p_t\}$. By Cauchy theorem, $t\le 4$. If $t=4$, then $c_{p_i}=1$ for $i=1, \cdots, 4$. Thus $G$ has a normal cyclic subgroup $N \cong C_{p_1p_2p_3p_4}$. We get a contradiction for $|C(C_{p_1p_2p_3p_4})|>5$. If $t=3$, from equation \ref{eq:1}, we get that $c_{p_1}=1$ or $2$, and $c_{p_2}=c_{p_3}=1$. Let $V, W \lhd G$ such that $V\cong C_{p_2}$ and $W\cong C_{p_3}$. Thus $|C(VW)|=4$, which implies that $c_{p_1}=1$. Let $U\lhd G$ such that $U\cong C_{p_1}$. Note that $UVW\cong C_{p_1p_2p_3}$ and $|C(C_{p_1p_2p_3})|>5$. This is impossible. So we get the following two cases:

	Case 1. $t=2$. Let $|G|=p^aq^b$. Let $P \in Syl_p(G)$ and $Q \in Syl_q(G)$. Suppose that neither $P$ nor $Q$ are cylcic subgroups of prime order. Then $|C(P)|\ge 3$ and $|C(Q)|\ge 3$. Thus $P \cup Q$ contains all the $5$ cyclic subgroups of $G$, and $G=P \cup Q$. This is impossible. So we can assume that $P \cong C_p$, and $|G|=pq^b$. 
	
	If $b=1$, by symmetry, we can assume $P \lhd G$. Note $|C(C_{pq})|=4$. We get that $Q$ is not normal in $G$. So $c_q=|G:N_G(Q)|=p$. Then $p=3$, and $q=2$. We get $G \cong S_3$.
	
	Now assume $b > 1$. 
	If $P$ is not normal in $G$, then $c_p=|G:N_G(P)|\ge 1+p$. We get $p=2$ and $c_p=3$. This imples $b=1$, a contradiction. Hence $P \lhd G$ and $c_p=1$. Clearly, $|C(Q)|\le 3$. By lemma \ref{3C}, $Q\cong C_{q^2}$.
	Now we get $c_{q^2}=|G:N_G(Q)|\ge 3$, contrary to $|C(G)|=5$.

	Case 2. $t=1$. Then $|G|=p^a$. If $p\ge 5$, then each cyclic subgroup of $G$ must be normal in $G$ for $|C(G)|=5$, which implies $G$ is a Dedekind group. Thus $G$ is abelian $p$-groups. Note  that $|C_p(C_p\times C_p)|=p+1>5$. It follows that $G$ is cyclic, and then $G \cong C_{p^4}$. 
	
	Next assume $p=3$. 
	Note $|C(C_3 \times C_3)|=5$. Thus $G$ can not have a proper subgroup isomorphic to $ C_3 \times C_3$. If $G \not \cong C_3\times C_3$, then all the subgroup of order $9$ are cyclic, which implies $G$ is cyclic by \cite[Satz 3.8.4]{Hup}. So we get $G \cong C_3\times C_3$ or $ C_{3^4}$.
	
	Finally, we consider the case that $p=2$. Let $|G|=2^n$. Since $|C(G)|=5$, we get $exp(G)=2^s\le 2^4$. 	
	If $s=1$, then $G$ is an elementary abelian $2$-group of order $2^n$ and $|C(G)|=2^n$, a contradicton. If $s =3$,  let $x\in G$ such that $|x|=2^3$. Since $|C(\cy{x})|=4$,
	there is only one cyclic subgroup in $G-\cy{x}$. Hence all the elements in $G-\cy{x}$ are contained in a cyclic subgroup and have the same order, $2^r$ say. So we get $|G-\cy{x}|=\phi(2^r)$, and 
	$2^n-8=2^{r-1}$, where $r \le 3$. This is impossible.
	
	Clearly, if $s=4$, then $G\cong C_{2^4}$. We need only to consider the case that $s=2$. Let $x \in G$ with $|x|=4$. 
	Hence we can get exactly two nontrivial cyclic subgroups from $G-\cy{x}$. Let $2^r$ and $2^2$ be the order of the two cyclic subgroups, where $r,s \le 2$. Then 
	$2^n-4=2^{r-1}+2^{s-1}$.
	So we get $n=3$ and $r,s=2$.
	Thus $G =Q_8$.
\end{proof}

Combining lemma \ref{3C}, \ref{4C} and \ref{5C}, we get the following.
\begin{theorem}
	For a finite group $G$, $|C(G)|\le 5$ if and only $G$ is a subgroup of $C_{p^4}$ or $G \cong S_3, Q_8, C_3\times C_3$, or $ C_{pq}$, where $p$ and $q$ are two different primes.
\end{theorem}

\end{document}